\documentclass[11pt]{amsart}
\usepackage{amssymb, latexsym, amsmath, amsfonts, tikz}
\usepackage{hyperref, enumitem, fancyhdr}

\newtheorem{thm}{Theorem}[section]

\newtheorem{prop}[thm]{Proposition}
\theoremstyle{definition}

\theoremstyle{remark}

\numberwithin{equation}{section}
\theoremstyle{remark}

\newcommand{\mbb}{\mathbb}
\newcommand{\ra}{\rightarrow}

\newcommand{\pa}{\partial}
\newcommand{\ov}{\overline}
\newcommand{\sm}{\setminus}
\newcommand{\ep}{\epsilon}
\newcommand{\no}{\noindent}
\newcommand{\al}{\alpha}
\newcommand{\Om}{\Omega}
\newcommand{\cal}{\mathcal}
\newcommand{\ti}{\tilde}
\newcommand{\la}{\lambda}

\newcommand{\de}{\delta}

\newcommand{\ga}{\gamma}
\newcommand{\Ga}{\Gamma}

\newcommand{\Si}{\Sigma}

\begin{document}
\title{Notes on the boundaries of quadrature domains}
\thanks{The author was supported in part by a UGC--CAS Grant}
\author{Kaushal Verma}

\address{Kaushal Verma: Department of Mathematics, Indian Institute of Science, Bangalore 560 012, India}
\email{kverma@math.iisc.ernet.in}

\begin{abstract}
We highlight an intrinsic connection between classical quadrature domains and the well-studied theme of removable singularities of analytic sets in several complex variables. Exploiting this connection provides a new framework to recover several basic properties of such domains, namely the algebraicity of their boundary, a better understanding of the associated defining polynomial and the possible boundary singularities that can occur.
\end{abstract}

\maketitle

\section{Introduction}
\noindent Recall that a bounded domain $\Om \subset \mbb C$ is called a (classical) quadrature domain for the test class $L^1_{\cal O}(\Om)$ of integrable holomorphic functions on $\Om$ if there exist finitely points $q_1, q_2, \ldots, q_m \in \Om$, positive integers $n_1, n_2, \ldots, n_m$ and constants $c_{j \al}$ such that
\begin{equation}
\int_{\Om} f \; dA = \sum_{j = 1}^{m} \sum_{\al = 0}^{n_j - 1} c_{j \al} f^{(\al)}(q_j)
\end{equation}
for every $f \in L^1_{\cal O}(\Om)$. Here $dA$ is the standard area measure on the plane. This condition, which is also called the quadrature identity, can be interpreted as saying that integration as a functional on the given test space can be represented as a finite linear combination of point masses and some of their derivatives. The points $q_1, q_2, \ldots, q_m$ are called the nodes and the order of the quadrature domain $\Om$ is the cardinality of the support of these masses counting multiplicities, i.e., $N = N(\Om) = n_1 + n_2 + \ldots + n_m$. A prototypical example of a quadrature domain is the disc $B(a, r) \subset \mbb C$ for which the associated quadrature identity is the mean value property of holomorphic functions and this implies that the disc has order one. Belying all initial hunches about the stringency of  (1.1), quadrature domains possess a number of remarkable properties and perhaps the most intriguing one is that the boundary $\pa \Om$ is algebraic (with the possible exception of finitely many points). More precisely,
\[
\pa \Om = \{ z \in \mbb C : P(z, \ov z) = 0 \} \setminus \{\text {possibly finitely many points}\}
\]
where $P(z, w)$ is an irreducible polynomial in the variables $(z, w) \in \mbb C^2$ satisfying
\[
P(z, w) = \ov{P(\ov w, \ov z)}.
\]
Being a quadrature domain is equivalent to the existence of a meromorphic function (the Schwarz function) $S(z)$ on $\Om$ that extends continuously to the boundary and equals $\ov z$ on $\pa \Om$ and the intrinsic utility of $S(z)$ lies in the fact that it enables the compact Riemann surface corresponding to $P(z, w)$ to be identified with the Schottky double of $\Om$. The construction of $S(z)$ shows that its poles are precisely at the nodes $q_i$ ($1 \le i \le m$) with corresponding multiplicity $n_i$ ($1 \le i \le m$).

\medskip

The purpose of these notes is to provide an alternate method to obtain these well known results by using an intrinsic link between quadrature domains and the theme of removable singularities in several complex variables. Indeed, the graph of $S(z)$ is a pure $1$-dimensional analytic set in $\Om \times \mathbb P^1$ whose boundary is contained in the totally real manifold
$M = \{ (z, w) \in \mbb C^2 : w = \ov z\}$ and since points of $\Om$ correspond to points on the graph of $S$, it follows that understanding the boundary of $\Om$ reduces to understanding the behavior of the analytic set near $M$. Various aspects of this situation, namely an analytic set with boundary in a totally real manifold, have been studied but what is perhaps most relevant is the Alexander--Becker theorem (\cite{Al}, \cite{Be}) which asserts that such analytic sets admit analytic continuation across totally real obstructions as analytic sets. This approach has several benefits. To start with, it yields an alternate proof of the classical Aharonov--Shapiro algebraicity result \cite{AS}, a direct consequence of which is the aforementioned fact that quadrature domains have algebraic boundaries. It is also possible to recover Gustafsson's result \cite{Gu} which asserts that
\[
P(z, \ov z) = (x^2 + y^2)^N + \text{lower order terms},
\]
$N$ being the order of $\Om$. All these matters are discussed in Section $2$.

\medskip

The next natural step is to understand the possible singularities that $\pa \Om$ possesses. This was done by Sakai \cite{S1} who considered the following local question: take an arbitrary open set $D$ in the unit disc $B(0, 1)$ such that the origin is contained in $\pa D$ and let $\Ga = \pa D \cap B(0, 1)$ and suppose that $D$ admits a Schwarz function in the sense that there is a continuous function $S$ defined on $D \cup \Ga$ such that $S$ is holomorphic on $D$ and $S(z) = \ov z$ on $\Ga$. If the origin is not an isolated point of $\Ga$, what can be said about the structure of $\Ga$? Sakai's theorem gives a complete classification of the various possibilities for $\Ga$ and the main point of this analysis is a detailed study of the  holomorphic function $z S(z)$ (which equals $z \ov z = \vert z \vert^2$ on $\pa \Om$). One of the main difficulties is to show that
\begin{equation}
\lim_{z \in D, z \rightarrow 0} \vert S(z) / z \vert = 1.
\end{equation}
Working with the graph of $S$ allows all this to be interpreted more geometrically as follows. First, $zS(z)$ is exactly the restriction of the quadratic projection $(z, w) \mapsto z^2 + w^2$ to the graph of $S$ after a $\mathbb C$-linear change of variables and it is well known that quadratic projections from analytic sets provide finer details about their boundary behaviour near totally real obstructions than linear ones. This clarifies the choice of $z S(z)$. Second, the limiting ratio of $\vert S(z)/ z\vert$ is in fact a unit vector in the (complex) tangent cone at the origin to the analytic set that extends the graph of $S$. The tangent cone is a finite union of complex lines, each of which intersects $M$ in either a point or a real line. The assumption that the origin is a non-isolated point in $\Ga$ forces the tangent cone to consist of only those complex lines that intersect $M$ in a real line. Such complex lines were called semi-real in Chirka's comprehensive study of the boundary regularity of $1$-dimensional analytic sets near totally real obstructions -- see \cite{Ch}. That the  limit in (1.2) exists and equals $1$ is a consequence of the fact that the tangent cone consists entirely of semi-real lines. These two remarks set the tone for an account of Sakai's theorem in the framework of local analytic geometry and the discerning reader will have no difficulty in recognizing the influence of some of the basic ideas in \cite{Ch} in what is presented here. However, the full force of these techniques will not be needed. We give a complete proof of Sakai's theorem and also mention some simplifications in Sections $3, 4$ and $5$.

\section{Quadrature domains have algebraic boundaries}

\noindent Using a pole elimination technique, Aharanov-Shapiro \cite{AS} proved the following theorem, which we state in their own words:

\medskip

{\sf Let $D$ be a bounded plane domain. Suppose $f$ and $g$ are holomorphic in $D$ except for finitely many polar singularities. Suppose moreover that $f$ and $g$ are continuously extendible to $\pa D$ and there take only real values.Then there is a non-trivial polynomial $P(X, Y)$ such that $P(f(z), g(z)) = 0$. Moreover, $P$ can be taken to have real coefficients, and be irreducible over the complex field.}

\medskip

\begin{proof}
By regarding $f, g$ as holomorphic maps to $\mbb P^1$, it follows that
\[
\phi(z) = \left( \frac{f(z) - i}{f(z) + i}, \frac{g(z) - i}{g(z) + i}  \right)
\]
defines a proper holomorphic mapping from $D$ into $(\mbb P^1 \times \mbb P^1) \sm \mbb T^2$, where $\mbb T^2 \subset \mbb R^2$ is the standard $2$-torus. Indeed, since $f, g$ are both real valued on $\pa D$, the map $\phi$ extends continuously to $\pa D$ and hence for any sequence $\{p_j\}$ that clusters only at $\pa D$, each component of $\phi(p_j)$ has unimodular limiting values, which exactly means that $\phi(p_j)$ clusters at a point on $\mbb T^2$. This implies that $\phi$ is proper. By Remmert's theorem, $A = \phi(D)$ is an irreducible one dimensional analytic set in $(\mbb P^1 \times \mbb P^1) \sm \mbb T^2$. Now form the reflection $A^{\ast}$ of $A$. To do this, note that reflection in $\mbb T^2$ is given by the involution $\sigma(z_1, z_2) = (1/\ov z_1, 1/\ov z_2)$ each of whose components defines the standard inversion in the unit circle. Evidently, $\sigma$ extends to a conjugate holomorphic involution on $\mbb P^1 \times \mbb P^1$ and hence defines a real structure on it. Thus, as a set $A^{\ast} = \{\sigma(a) : a \in A\}$ and moreover, if in a neighbourhood $U$ of one of its points $a$, $A$ is defined by the holomorphic functions $\{f_j\}$, then in $U^{\ast} = \sigma(U)$ (which is a neighbourhood of $a^{\ast} = \sigma(a)$), $A^{\ast}$ is defined by the holomorphic functions $\{ \ov{f_j \circ \sigma(z_1, z_2)} =  \ov{f(1/\ov z_1, 1/\ov z_2)}\}$. Then $A \cup A^{\ast}$ is an irreducible analytic set in $(\mbb P^1 \times \mbb P^1) \sm \mbb T^2$. By the Alexander--Becker theorem, $\mathcal A = \ov{A \cup A^{\ast}}$ is an analytic set in $\mbb P^1 \times \mbb P^1$. Recall that the Segre embedding $s : \mbb P^1 \times \mbb P^1 \ra \mbb P^3$ defined by
\[
s([x_0 : x_1], [y_0 : y_1]) = [z_0 : z_1 : z_2 : z_3]
\]
where $z_0 = x_0 y_0, z_1 = x_0 y_1, z_2 = x_1 y_0, z_3 = x_1 y_1$ exhibits $\mbb P^1 \times \mbb P^1$ as a smooth quadric in $\mbb P^3$ given by
\[
z_0 z_3 - z_1 z_2 = 0.
\]
Thus $s(\mathcal A) \subset \mbb P^3$ is analytic and hence algebraic. Any homogeneous polynomial $Q = Q(z_0, z_1, z_2, z_3)$ of degree $k$ that vanishes on $s(\mathcal A)$ must have the form
\[
Q = \sum_{\alpha, \beta, \gamma, \delta \geq 0} c_{\alpha \beta \gamma \delta} z_0^{\alpha} z_1^{\beta} z_2^{\gamma} z_3^{\delta}
\]
where $\alpha + \beta + \gamma + \delta = k$. Rewriting this in terms of the $x_i$'s and $y_i$'s, the polynomial has the form
\[
Q = \sum_{\alpha, \beta, \gamma, \delta \geq 0} c_{\alpha \beta \gamma \delta} x_0^{\alpha + \beta} x_1^{\gamma + \delta} y_0^{\alpha + \gamma} y_1^{\beta + \delta}
\]
which shows that the pullback $Q \circ s$ is bihomogeneous, i.e., homogeneous in the $x_i$'s and $y_i$'s separately.  Since $\cal A \subset \mbb P^1 \times \mbb P^1$ is purely one dimensional, it must be principal and hence there is an irreducible bihomogeneous polynomial $P = P(x_0, x_1, y_0, y_1)$ that defines it and in particular vanishes on $A = \phi(D)$. If $X = x_0/x_1$ and $Y = y_0/y_1$ are the affine coordinates on the two copies of $\mbb P^1 \sm \{ \infty\}$, it follows that $P(X, 1, Y, 1) \in \mbb R[X, Y]$ vanishes on $A \cap \mbb C^2$ and hence there is a polynomial (which we still denote by $P$) such that
\begin{equation}
P(f(z), g(z)) = 0
\end{equation}
 for $z \in \ov \Om$ away from the poles of $f, g$. In effect, this last step amounts to undoing the linear fractional transformation in each component of $\phi$ and working essentially with $z \mapsto (f(z), g(z))$ which takes values in $\mbb R^2$ on $\pa D$. To complete the proof, note that $\cal A$ is invariant under $\sigma$ (by construction) and hence the polynomial
 $P$ from (2.1) must be invariant under the involution that fixes $\mbb R^2 \subset \mbb C^2$, i.e., $(z, w) \mapsto (\ov z, \ov w)$. Hence
\[
\ov{P(X, Y)} = {P\left(\ov X, \ov Y\right)}
\]
which implies that $P$ has real coefficients. The same arguments apply in the following two cases: (i) $f, g$ have finitely many polar singularities on $\pa D$, i.e., points $\zeta \in \pa D$ such that $\vert f(z) \vert$ or $\vert g(z) \vert \ra +\infty$ as $z \in \ov D \sm \{\zeta\}$ approaches $\zeta$ and (ii) $D$ is unbounded. These observations were also mentioned by them.

\end{proof}

\noindent The conclusion that quadrature domains have algebraic boundaries is now clear. Take the  Schwarz function $S(z)$ corresponding to $\Om$ and apply the Aharanov--Shapiro theorem to
\[
f(z) = (z + S(z))/2, \;g(z) = (z - S(z))/2i
\]
both of which are meromorphic in $\Om$ with finitely many poles and extend continuously up to $\pa \Om$. Since $S(z) = \ov z$ on $\pa \Om$ it follows that $f(z) =x$ and $g(z) = y$ there and hence
\[
\pa \Om \subset \{(x, y) : P(x, y) = 0\}.
\]
To obtain further properties of this polynomial, it is best to apply the aforementioned steps to the given quadrature domain $\Om$ from the very beginning, in which case they take on a particularly transparent form. To do this, first note that the complex linear change of coordinates
\begin{equation}
T(z, w) = (T_1(z, w), T_2(z, w)) = (z/2 + w/2, -iz/2 + iw/2) = (Z, W)
\end{equation}
transforms the totally real manifold $M$ into $\{ (Z, W) : \Im Z = \Im W = 0 \} \backsimeq \mbb R^2$. The involution $\tau(z, w) = (\ov z, \ov w)$ is the conjugate holomorphic reflection in $\mbb R^2$ and hence
\[
T^{-1} \circ \tau \circ T(z, w) = (\ov w, \ov z)
\]
defines the conjugate holomorphic reflection in $M$. Let $A = \text{Graph}(S) = \{ (z, S(z)) : z \in \Om\}$ which is an irreducible $1$-dimensional analytic set with boundary in $M$. The reflection $A^{\ast}$ of $A$ consists of points of the form $\{ (\ov{S(z)}, \ov z) : z \in \Om \}$ and this is also analytic. Indeed, $A$ being the graph of the Schwarz function $S$ is defined by
\[
\phi(z, w) = w -  S(z) = 0
\]
and hence (by definition) $A^{\ast}$ is the zero locus of the holomorphic function
\[
\ov{\phi (T^{-1} \circ \tau \circ T(z, w) )} = \ov{\phi(\ov w, \ov z)} = z - \ov{S(\ov w)}.
\]
Hence $A \cup A^{\ast} \subset (\mbb P^1 \times \mbb P^1) \sm M$ is a one dimensional analytic set and therefore the Alexander--Becker theorem shows that $\cal A = \ov{A \cup A^{\ast}}$ is an irreducible analytic set in $\mbb P^1 \times \mbb P^1$ which is defined by an irreducible bihomogeneous polynomial, say $P(x_0, x_1, y_0, y_1)$. Here, as before, the two copies of $\mbb P^1$ are assumed to have projective coordinates $[x_0:x_1]$ and $[y_0:y_1]$ respectively. In affine coordinates $z = x_0/x_1, w = y_0/y_1$, $\cal A \cap \mbb C^2$ is defined by $P(z, 1, w, 1)$ which we will write as $P(z, w)$. Since $\cal A$ is invariant under reflection in $M$, i.e., $T^{-1} \circ \tau \circ T$, the same must hold for $P(z, w)$. Hence
\begin{equation}
P(z, w) = \ov{P(\ov w, \ov z)}
\end{equation}
which shows that if  $P(z, w) = \sum c_{\alpha  \beta} z^{\alpha} w^{\beta}$ then $c_{\alpha \beta} = \ov c_{\beta \alpha}$. Let $\pi_1, \pi_2$ be the projections from $\cal A$ to the factors $\mbb P^1$ (with coordinates $[x_0:x_1]$) and $\mbb P^1$ (with coordinates $[y_0:y_1]$) respectively, as shown:

\begin{center}
\begin{tikzpicture}[every node/.style={midway}]
\matrix[column sep={4em,between origins},
        row sep={2em}] at (0,0)
{ \node(R)   {$\mathcal A$}  ; & \node(S) {$\mathbb P^1$}; \\
  \node(R/I) {$\mathbb P^1$};                   \\};
\draw[<-] (R/I) -- (R) node[anchor=east]  {$\pi_1$};
\draw[->] (R/I) -- (S) node[anchor=north]  {$\psi$};
\draw[->] (R)   -- (S) node[anchor=south] {$\pi_2$};
\end{tikzpicture}
\end{center}

\noindent Both are evidently proper and hence by Remmert's theorem, $\pi_1(\cal A)$ and $\pi_2(\cal A)$ are analytic sets in the respective copies of  $\mbb P^1$. Furthermore, none of the projections collapse $\cal A$ to a point because both $\pi_1(\cal A), \pi_2(\cal A)$ contain $\Om$ by construction. Hence both $\pi_1$ and $\pi_2$ are surjective which then implies that $\cal A$ is a finite branched cover over both factors, i.e., $\cal A \subset \mbb P^1 \times \mbb P^1$ is a holomorphic correspondence. Now note that the fiber $\pi_1^{-1}\left([0:1]\right)$ does not intersect $A$ (since it is the graph of $S$ over a bounded domain $\Om$) and hence
\[
\cal A \cap \pi_1^{-1}\left([0:1]\right) = \ov {A^{\ast}} \cap \pi_1^{-1}\left([0:1]\right)
\]
where both are finite sets. By the definition of $A^{\ast}$, it follows that $\ov {A^{\ast}} \cap \pi_1^{-1}\left([0:1]\right)$ as a set consists precisely of the conjugates of the nodes  $q_1, q_2, \ldots, q_m$ and hence has cardinality $N = N(\Om)$, the order of $\Om$, when counted with multiplicity. It can be checked that $\pi_2$ has exactly the same properties. Thus the projection multiplicity of $\cal A$ as a branched cover over each factor of $\mbb P^1$ is $N$. It follows that the degree of $P(z, w)$ in both $z$ and $w$ is $N$. This has two consequences:

\medskip

For the first one, let
\[
P(z, w) = w^N P_0(z) + w^{N-1}P_1(z) + \ldots + P_N(z)
\]
where the coefficients $P_i(z)$ are polynomials of degree no more than $N$. Therefore
\[
P(z, S(z)) = (S(z))^N P_0(z) + (S(z))^{N-1} P_1(z) + \ldots + P_N(z) = 0
\]
for all $z \in \Om$ which can be rewritten as
\[
P_0(z) = -  P_1(z)/S(z) -  \ldots -  P_N(z)/(S(z))^N
\]
away from a finite set in $\Om$. The expression on the right vanishes at the nodes $q_1, q_2, \ldots, q_m$ with corresponding multiplicities $n_1, n_2, \ldots, n_m$ since the polynomials $P_i(z)$ ($1 \le i \le N$) are all bounded on $\Om$ and the Schwarz function $S(z)$ has poles precisely at the nodes. Hence the degree of $P_0(z)$ is $n_1 + n_2 + \ldots + n_m = N$ which shows that
\[
P(z, w) = c_{NN} z^N w^N + \text{terms of lower order}.
\]
The coefficient $c_{NN}$ is real because of (2.3). The boundary $\pa \Om$ is contained in $\{P(z, \ov z) = 0\}$ and therefore
\[
P(z, \ov z) = c_{NN} (x^2 + y^2)^N + \text{terms of lower order},
\]
where $z = x+ iy$. This was first obtained by Gustafsson (see \cite{AS} and \cite{Gu} for further information) and a corollary is the the well known result that the ellipsoidal region
\[
\left\{ (x, y) \in \mbb R^2 : x^2/a^2 + y^2/b^2 < 1  \right\}
\]
 (where $a \not= b$) is not a quadrature domain.

\medskip

The second consequence relates to the valence of the Schwarz function in $\Om$, i.e., for a given $w \in \mbb C$, how many zeros, counting multiplicities, does $S(z) - \ov w$ have in $\Om$?
When $w \in \mbb C \sm \ov \Om$, the fiber $\pi_1^{-1}(w)$ does not intersect $A$ as explained above.  But $\pi_1$ has multiplicity $N$, which means that all points in $\cal A$ that lie over $w$ must come from $A^{\ast}$. By the definition of $A^{\ast}$, this implies that $S(z) - \ov w$ has $N$ roots in $\Om$. If $w \in \Om$, then exactly one point in the fiber $\pi_1^{-1}(w)$ lies in $A$, namely $S(w)$. Therefore the remaining $N-1$ points must belong to $A^{\ast}$ and hence $S(z) - \ov w$ has $N-1$ roots in $\Om$. These were first obtained by Avci and different proofs were provided by Shapiro -- see \cite{Sh1}, \cite{Sh2} for a leisurely account. The remaining possibility is when $w \in \pa \Om$. This will require additional information about the possible singularities of $\pa \Om$ and we will return to this question then.


\section{Regularity of the boundary: Sakai's theorem}

\noindent The boundary of a quadrature domain cannot have singularities with an arbitrary prescribed complexity. In fact, Sakai's theorem gives a complete classification of the possible singularities that can occur. Loosely speaking, apart from a degenerate situation, there are only two types of singular points -- a double point that is formed by the tangency between two smooth real anaytic branches and an inward pointing cusp, and both actually occur as examples show. More precisely, let $D \subset B(0, 1)$ be an arbitrary open set such that the origin lies on $\pa D$. Let $\Ga = \pa D \cap B(0, 1)$ and assume that the origin is not an isolated point of $\Ga$. Suppose that there is a continuous function $S$ on $D \cup \Ga$ that is holomorphic on $D$ and equals $\ov z$ on $\Ga$. Sakai's theorem, which is stated below in essentially his own terms, in fact applies to this local situation and classifies the possible structure of $\Ga$.

\medskip

\noindent {\sf One of the following must occur for a small disc $B(0, \delta)$:

\begin{enumerate}

\item[(1)] $D \cap B(0, \delta)$ is simply connected and $\Ga \cap B(0, \delta)$ is a regular real analytic simple arc passing through $0$.

\medskip

\item[(2a)] $\Ga \cap B(0, \delta)$ determines uniquely a regular real analytic simple arc passing through $0$ and $\Ga \cap B(0, \de)$ is an infinite proper subset of the arc that accumulates at $0$ or is the whole arc. In this case $D \cap B(0, \delta)$ is equal to $B(0, \delta) \sm \Ga$.

\medskip

\item[(2b)] $D \cap B(0, \delta)$ consists of two simple connected components $D^+$ and $D^-$. Furthermore, each of $\pa D^{\pm} \cap B(0, \delta)$ are distinct regular real analytic simple arcs passing through $0$ and they are tangent to each other to $0$.

\medskip

\item[(3)] $D \cap B(0, \delta)$ is simply connected and $\Ga \cap B(0, \delta)$ is a regular real analytic simple arc except for a cusp at $0$ that points inwards into $D \cap B(0, \delta)$. It is a very special one. There is a holomorphic function $T$ defined on the closure of $B(0, \delta)$ such that $T$ has the following properties: it has a zero of order two at $0$, it is injective on the closure of the upper half disc in $B(0, \delta)$ (namely $H^+ = \{ \tau \in B(0, \delta) : \Im \tau > 0 \}$), and finally $T(\ov H^+) \subset D \cup \Ga$ along with $\Ga \cap B(0, \delta) \subset T\left( (-\ep, \ep) \right)$.

\end{enumerate}
}

To rephrase $(3)$ -- the map $T : \ov H^+ \rightarrow D \cup \Ga$ is biholomorphic as a map from $H^+$ onto $D$, it extends holomorphically across $\pa H^+$ and maps the open  interval $(-\ep, \ep)$ (which is a subset of $\pa H^+$) onto $\Ga$ near the origin. The application to quadrature domains follows by taking $S$ to be the corresponding Schwarz function. The main step in Sakai's proof is (1.2), which is done in two parts, namely
\[
\limsup_{z \in D, z \rightarrow 0} \vert S(z) / z \vert \leq 1 \;\; \text{and} \;\; \liminf_{z \in D, z \rightarrow 0} \vert S(z) / z \vert \geq 1.
\]
The first bound is a direct consequence of a theorem of Fuchs (see Sakai \cite{S1}) while the lower bound is more technically involved. In this section, we show how to formulate and reprove Sakai's theorem in more geometric terms. In particular, it is possible to give a different proof of the lower bound above.

\begin{prop}
In the situation described above,
\[
\lim_{z \in D, z \ra 0} \vert S(z)/z \vert = 1.
\]
\end{prop}

\begin{proof}
To start with, note that
\[
A = \text{Graph}(S) \sm M
\]
is a non-empty closed analytic set in $(B(0, 1) \times \mbb C) \sm M$ whose boundary $\pa A$ is contained in $M$. It is non-empty as otherwise $\text{Graph}(S)$ which is an analytic set will be entirely contained in the totally real manifold $M$ and this cannot happen. That $A$ is closed follows from the continuity of $S$. However, $A$ may not be irreducible since $D$ may have several components. Form the reflection $A^{\ast}$ as explained before. Then $A \cup A^{\ast} \subset (B(0, 1) \times \mbb C) \sm M$ is a closed analytic set which is invariant under reflection in $M$ and the Alexander--Becker theorem shows that $\cal A = \ov{A \cup A^{\ast}}$ is analytic in $B(0, 1) \times \mbb C$. Note that $(0, 0) \in \cal A$. Recall that the (complex) tangent cone to $\cal A$ at the point $(0, 0)$, which will be denoted by $C(\cal A, (0, 0))$, consists of all vectors $v  \in \mbb C^2$ that are expressible as
\[
v = \lim_{j \ra \infty} \la_j a_j
\]
where $a_j \in \cal A$ converges to $(0, 0)$ and $\la_j \in \mbb C$. Since $\cal A$ is analytic, $C(\cal A, (0, 0))$ is the union of finitely many complex lines through the origin $(0, 0)$. The relative position of these lines with respect to $M$ governs the behaviour of $\cal A$ (and hence $A$) near $M$. A complex line in $\mbb C^2$ through the origin, other than the $w$-axis, is of the form $\zeta \mapsto \zeta(1, m)$ for some $m \in \mbb C$. The former, which can be parametrized as $\zeta \mapsto \zeta (0, 1)$ intersects $M$ exactly at the origin. The intersection of the latter with $M$ depends on $m$. The case $m = 0$ corresponds to the $z$-axis which intersects $M$ exactly at the origin. Now suppose that $m \not= 0$ and that there is a nonzero $\la_0$ such that $(\la_0, \la_0 m) \in M$. Then
\[
\la_0 m = \ov \la_0
\]
and thus $\vert m \vert = 1$. Conversely, if $m = e^{-2i \alpha}$ for some $\alpha$, then $(\la_0, \la_0 m) \in M$ where $\la_0 = e^{i \alpha}$. Hence, complex lines with $\vert m \vert = 1$  are precisely the ones that intersect $M$ in a real line, which then turns out to be the real span of the vector $(\la_0, \la_0 m)$. These are the semi-real lines in Chirka's terminology. All other values of $m$ correspond to complex lines, that will be called complex transversals, which intersect $M$ precisely at the origin. The property of being complex transversal is clearly an open condition.

\medskip

Since $A \subset \cal A \subset B(0, 1) \times \mbb C$, there are finitely many irreducible components, say $A_1, A_2, \ldots, A_k$ of $A$ that cluster at the origin and in fact for each $A_j$, there is a unique irreducible component $\cal A_j \subset \cal A$ such that $A_j \subset \cal A_j$. There are of course finitely many irreducible components of $\cal A$ as well by its analyticity near the origin. Each $\cal A_j$ is tangent to a unique line in $C(\cal A, (0, 0))$ and hence so is $A_j$. Now note that for $\zeta \in \pa D$, $\zeta \not= 0$,
\[
\lim_{z \in D, z \ra \zeta} \vert S(z)/z \vert = \lim_{z \in D, z \ra \zeta} \vert \ov \zeta / \zeta \vert  = 1
\]
and by the continuity of $S$ on $\ov D$, it follows that for a fixed $\delta \in (0, 1)$,
\[
\vert S(z)/ z \vert \le C \vert z \vert^{-1}
\]
in $D \cap B(0, \delta)$. By Fuchs' theorem,
\[
\limsup_{z \in D, z \ra 0} \vert S(z)/z \vert \le 1.
\]
This means that $C(A, (0, 0))$, which consists of limiting positions of the complex lines through the origin in the direction of $(z, S(z))$, cannot contain the $w$-axis (the vertical direction). In fact, $C(A, (0, 0))$ does not contain any complex transversals. To see this, suppose that $\cal L$ is one such line. Then $\cal L = \la(1, m)$ where $\vert m \vert \not= 1$ and let $A_1 \subset A$ be the component which is tangent to it. By the definition of the tangent cone, there is a small conical neighbourhood (formed by taking the union of $\cal L$ with its neighbouring complex transversals) that contains $A_1$. Then $A_1$ clusters only at the origin and since it is one dimensional, $\ov A_1 = A_1 \cup \{(0, 0)\}$ is a closed analytic set by the Remmert-Stein theorem. Since $\cal L$ is not the $w$-axis, the projection $\pi(z, w) = z$ is locally proper near the origin when restricted to $\ov A_1$. That is, there is a neighbourhood $U = U_1 \times U_2$ of $(0, 0)$ such that
\[
\pi : \ov A_1 \cap U \ra U_1 \subset \mbb  C
\]
is proper and hence surjective. Moreover, $\pi(A_1) = U_1 \sm \{0\}$. But this implies that $0 \in \Ga$ is an isolated point, which is a contradiction. It follows that all lines in $C(A, (0, 0))$ are semi-real.

\medskip

Let $q(z, w) = z^2 + w^2$. Then, with $T$ as in (2.2),
\[
Q(z, w) = q \circ T(z, w) = T_1^2(z, w) + T_2^2(z, w) = zw
\]
and
\[
\vert T(z, w) \vert^2 = \vert T_1(z, w) \vert^2 + \vert T_2(z, w) \vert^2 = \left (\vert z \vert^2 + \vert w \vert^2 \right)/2 = \vert (z, w) \vert^2/2.
\]
On semi-real lines, which are parametrized as $\la \mapsto \la(\la_0, \la_0 m)$ with $\la_0, m$ as above, it follows that
\[
\vert Q(\la, \la_0, \la \la_0 m) = \vert \la^2 \la_0^2 m \vert = \vert \la \la_0 \vert^2
\]
and
\[
\vert T(\la \la_0, \la \la_0 m) \vert ^2= \left( \vert \la \la_0  \vert^2 + \vert \la \la_0 m \vert^2 \right)/2 = \vert \la \la_0 \vert^2
\]
since $\vert m \vert = 1$. That is, $\vert Q \vert = \vert T \vert^2$ on $C(A, (0, 0))$. Since $A$ is asymptotic to $C(A, (0, 0))$ near the origin, it is expected that
\begin{equation}
\lim_{(z, w) \in A, (z, w) \ra 0}  \frac{\vert Q(z, w) \vert}{\vert T(z, w) \vert^2} = 1.
\end{equation}
Indeed, note that
\[
\frac{\vert Q(z, w) \vert}{\vert T(z, w) \vert^2} = \frac{\vert Q(z, w) \vert}{\vert T_1(z, w) \vert^2 + \vert T_2(z, w) \vert^2} = \frac{2 \vert zw \vert}{\vert z \vert^2 + \vert w \vert^2}
\]
and
\[
Q \left( \frac{(z, w)}{\vert(z, w)\vert}\right) = Q \left( \frac{(z, w)}{\sqrt{\vert z \vert^2 + \vert w \vert^2}}\right) = \frac{zw}{\vert z \vert^2 + \vert w \vert^2}
\]
and hence
\[
\frac{\vert Q(z, w) \vert}{\vert T(z, w) \vert^2} = 2 Q \left( \frac{(z, w)}{\vert(z, w)\vert}\right)
\]
for all $(z, w) \in \mbb C^2$. For $(z, w) \in A$ such that $(z, w) \ra (0, 0)$, all the limiting values of the unit vectors $(z, w)/\vert(z, w) \vert$ lie in the set of unit vectors in $C(A, (0, 0))$. For such vectors $v$, $2 \vert Qv \vert = 2 \vert Tv \vert^2 = \vert v \vert^2 = 1$ where the first equality holds since $v$ is semi-real, the second holds for all $v$ by the definition of $T$ and finally, the third holds since $v$ is a unit vector. In the limit, we get
\[
\lim_{(z, w) \in A, (z, w) \ra 0}\frac{\vert Q(z, w) \vert}{\vert T(z, w) \vert^2} = \lim_{(z, w) \in A, (z, w) \ra 0} 2 Q \left( \frac{(z, w)}{\vert(z, w)\vert}\right) = 2Q(v) = 1.
\]
But points on $A$ are of the form $(z, S(z))$ for $z \in D$ and hence
\[
1 = \lim_{(z, w) \in A, (z, w) \ra 0}\frac{\vert Q(z, w) \vert}{\vert T(z, w) \vert^2} = \lim_{z \ra 0, z \in D}  \frac{2\vert z S(z) \vert}{\vert z \vert^2 + \vert S(z) \vert^2}
\]
which implies that $\vert S(z) / z \vert \ra 1$ as $z \ra 0$ in $D$.
\end{proof}

\medskip

Consider the projection $Q : \cal A \ra \mbb C$.  What is the restriction of $Q$ to $\cal A$? On $A$, it is exactly the holomorphic function $z \mapsto z S(z)$, while on $A^{\ast}$ (which consists of points of the form $(\ov{S(z)}, \ov z)$, it is $z \mapsto \ov{z S(z)}$ for $z \in D$. Then $Q^{-1}(0) \cap \cal A$ contains the origin as an isolated point as otherwise some component of $\cal A$ will be contained in the analytic set $\{zw = 0\}$ by the uniqueness theorem. This would then imply that the graph of $S(z)$ over some connected component of $D$ is either contained in the $w$-axis (i.e., it is vertical) or the $z$-axis (i.e., $S \equiv 0$ in that component). The former cannot happen and the latter forces $S$ to vanish on the boundary of this component. But then this violates the given fact that $S(z) = \ov z$ on $\pa D$. Hence, there is a neighbourhood $U$ of the origin $(0, 0)$ such that $Q : \cal A \cap U \ra Q(\cal A) \subset \mbb C$ is proper. We may assume that $Q(\cal A)$ is the unit disc $\{ \vert \eta \vert < 1\}$.

\medskip

Now $\cal A \cap M \cap U$ is a closed real analytic set in $M \cap U$ and in particular, it is semi-analytic. Therefore, it admits a  semi-analytic stratification as $\cal A \cap M \cap U = T_1 \cup T_0$ where $T_i$ ($i = 0, 1$) is a locally finite union of real analytic submanifolds of $M$ of dimension $i$. Furthermore, each component of $T_i$ is semi-analytic as a subset of $M \cap U$. By shrinking $U$, we may assume that  $\cal A \cap M \cap U$ has finitely many connected components. By refining this stratification, it is possible to assume that $(0, 0) \in T_0$ and hence each connected component of $\cal A \cap M \cap U$ is either a smooth real analytic arc or consists of finitely many smooth real analytic arcs ending at points in $T_0$. One such component contains the origin at which finitely many smooth real analytic arcs end. Let $\ga_1, \ga_2, \ldots, \ga_k$ be the arcs in $\cal A \cap M \cap U$ of which some will contain the origin as an end point. By shrinking $U$, we may assume that each of $\gamma_1, \gamma_2, \ldots, \gamma_k$ contain the origin as an end point.

\medskip

Let $\cal A \cap M \cap U = S_1 \cup S_2$ where
\[
S_1 = \{ a : \text{either} \; \ov A \;\text{or} \;\ov A^{\ast} \; \text{is an analytic set near} \; a \}
\]
and $S_2 = (\cal A \cap M \cap U) \sm S_1$. Then $S_1$ consists precisely of those points for which there is no need to adjoin $A^{\ast}$ (respectively $A$) to complete $A$ (respectively $A^{\ast}$) to obtain a closed analytic set. On the other hand, $S_2$ consists of those points near which the presence of $A^{\ast}$ (respectively $A$) is essential to complete $A$ (respectively $A^{\ast}$) to obtain a closed analytic set. In other words, near points in $S_2$, neither $A$ nor $A^{\ast}$ have analytic closures taken by themselves. Now suppose that $S_1 \cap \ga_1 \not= \emptyset$. Then by \cite{Ch2} (see Theorem 1 in Section 18.2 of Chapter 4 therein), it follows that $\ov A$ will be analytic near every point of $\ga_1$ and hence $\ga_1 \subset S_1$. In other words, $S_1$ contains the entire arc the moment it contains a point on it and therefore the set of arcs $\{\ga_l\}$ can be divided into two sets -- those that are contained in $S_1$ and those which do not intersect $S_1$. It is safe to ignore the former set of arcs since $\ov A$ (or $A^{\ast}$) is already analytic near them. It is the latter that needs to be analyzed. This is the collection of those arcs that are contained in $S_2 \subset M$. On $M$, the projection $Q$ is real valued and non-negative -- indeed, it is $z \mapsto z \ov z = \vert z \vert^2 \geq 0$. Therefore, the images of the arcs in $S_2$ are finitely many relatively closed connected subsets of the half-line $\{ \Re \eta \geq 0\}$, all of which contain the origin as a boundary point.

\medskip

Sakai's theorem gives a description of the possible configurations of $D$ (and hence $\Gamma$) and for this reason, it will suffice to restrict $Q$ to $\ov A \cap U$. Note that $Q : A \cap U \ra \mbb C$ is nonconstant and defines a proper projection onto its image $Q(A \cap U) \subset \{\vert \eta \vert < 1\} \sm \{ \Re \eta \geq 0\}$.

\begin{prop}
Let $\gamma$ be an arc in $S_2$ and suppose that $p \in \gamma \cap \ov A$. Then $\gamma \subset \ov A$.
\end{prop}

\begin{proof}
As noted earlier,
\[
Q : \cal A \cap U \ra \{ \eta : \vert\eta \vert < 1\} \subset \mathbb C
\]
is proper and hence has maximal generic rank. The singular locus of $\cal A$, {\text sng}($\cal A$) is zero-dimensional and hence by shrinking $U$ if needed, we can assume that the origin is possibly the only singularity of $\cal A$ in $U$. Since $Q$ is proper, the branch locus of $Q|_{\cal A \cap U}$, {\text br}($Q$) is discrete in $\cal A \sm \{(0, 0)\}$. Therefore, {\text br}$(Q) \;\cup$ {\text sng}($\cal A$) which is analytic, must be zero-dimensional. Thus $Q^{\prime} \not= 0$ on $\cal A \sm \{(0, 0)\}$. For the given arc $\gamma$, $\ov A \cap \gamma$ is clearly nonempty and closed in $\gamma$. It will suffice to show that $\ov A \cap \gamma$ is relatively open in $\gamma$. To do this, pick $a \in \ov A \cap \gamma$ and let $Q(a) = p > 0$. Let $U_a$ and $B(p, r)$ be small neighbourhoods of $a, p$ respectively such that
\begin{equation}
Q : \cal A \cap U_a \ra B(p, r)
\end{equation}
is a biholomorphism. Since $A \cap U_a \not= \emptyset$, the map $Q : A \cap U_a \ra B(p, r)$ is well defined. The positive $\Re \eta$-axis divides $B(p, r)$ into two components, say $B^{\pm}(p, r)$ which are defined by requiring $\pm \Im \eta > 0$ respectively. Suppose that $Q(A \cap U_a) \cap B^-(p, r) \not= \emptyset$. Let $q \in Q(A \cap U_a) \cap B^-(p, r)$ and let $Q_q^{-1}$ be a germ of the branched multivalued map $Q^{-1} : \{ \eta : \vert \eta \vert < 1\} \ra \cal A \cap U$ near $q$ such that $Q_q^{-1}(q) \in A \cap U_a$. Since (3.2) is a biholomorphism, $Q_q^{-1}$ can be analytically continued everywhere in $B^-(p, r)$; the analytic continuation will still be denoted by $Q_q^{-1}$. It follows that $Q_q^{-1}(B^-(p, r)) \subset A \cap U_a$. Also, the cluster set of $I = B(p, r) \cap \{\Im \eta = 0\}$ (which is the dividing diameter between $B^{\pm}(p, r)$) under $Q_q^{-1}$ is contained in $\gamma \cap U_a$. Since $\cal A$ and $\gamma$ are both smooth real-analytic near $a$, the Schwarz reflection principle shows that $Q_q^{-1}$ can be extended holomorphically across $I$. This extension agrees with the biholomorphism $Q^{-1} : B(p, r) \ra \cal A \cap U_a$. Thus, $\gamma = Q^{-1}(I)$ and since $Q_p^{-1}(B^-(p, r)) \subset A \cap U_a$, it follows that $\ov A \cap \gamma$ contains a relatively open arc (in $\gamma$) around $a$. This completes the proof.

\end{proof}

This shows that $\ov A \cap M$ consists of finitely many smooth real arcs that contain the origin as an end point. In particular, this rules out the possibility that $\ov A \cap M$ consists of countably many continua that cluster at the origin.

\medskip

Consequently, $Q(\pa A \cap S_2)$ consists of finitely many closed connected subsets of $\{\Re \eta \geq 0\}$ and therefore by shrinking $U$, there exists $\eta_0 > 0$ such that
\[
Q : A \cap U \ra \{\vert \eta \vert < \eta_0\} \sm \{\Re \eta \geq 0\}
\]
is proper. Let $\nu$ be the projection multiplicity of $Q|_{A \cap U}$.

\begin{prop}
$\nu = 1$ or $2$.
\end{prop}

\begin{proof}
Note that $Q|_{A \cap U}$ is the map $F(z) = zS(z)$ which is well defined for $z \in D$ near the origin and extends continuously up to $\partial D$. As observed earlier, $Q^{-1}(0) \cap \mathcal A$ contains the origin as an isolated point. Thus $F(z) = 0$ has a unique root near the origin, namely $z = 0$. In particular, there exists a $\delta > 0$ such that
\begin{equation}
F : D \cap B(0, \delta) \ra B(0, \eta_0) \setminus \{ \Re \eta \geq 0 \}
\end{equation}
is proper. The projection multiplicity $\nu$ is precisely the valency of $F$ as defined above. By the generalized argument principle,
\[
\nu =\frac{1}{2 \pi} \int_{\pa(D \cap B(0, \delta))} d \arg F(z) = \frac{1}{2 \pi} \int_{D \cap B(0, \delta)} d \arg F(z)
\]
where the second equality holds since  $F(z) = z \ov z = \vert z \vert^2 \geq 0$ on $\pa D \cap B(0, \delta)$. To estimate $\nu$, let $G(z) = S(z)/z$ so that $F(z) = z^2 G(z)$ and
\[
\arg F(z) = 2 \arg z + \arg G(z).
\]
For $0 < r < \delta$, define
\[
A(r) = \int_{D \cap \pa B(0, r)} d \arg G(z)
\]
and note that
\[
A(r) = \int_{D \cap \pa B(0, r)} d \arg F(z) - 2 \int_{D \cap \pa B(0, r)} d \arg z \geq 2 \pi \nu - 4 \pi.
\]
To get a better estimate for $A(r)$ as $r \ra 0$, choose $0 < \rho < \delta$ such that $1 - \ep < \vert G(z) \vert < 1 + \ep$ for $z \in D \cap B(0, \rho)$. By the Cauchy--Riemann equations,
\[
A(r) = \int_{D \cap \pa B(0, r)} d \arg G(z)  = \int_{D \cap \pa B(0, r)} \frac{\pa}{\pa s} \arg G(z) \; ds = \int_{D \cap \pa B(0, r)} \frac{\pa}{\pa r} \log \vert G(z) \vert \; r d \theta.
\]
Therefore, for $0 < \eta < \rho$,
\[
\int_{\eta}^{\rho} \frac{A(r)}{r} \; dr = \int_{D \cap \left( B(0, \rho) \setminus \ov{B(0, \eta)} \right)} \frac{\pa}{\pa r} \log \vert G(z) \vert \; d \theta \; dr.
\]
To interchange the order of integration in the last integral, it is essential to verify that $\pa / \pa r \log \vert G(z) \vert$ is integrable on $D \cap (B(0, \rho) \setminus \ov{B(0, \eta)})$. But
\[
\left \vert \frac{\pa}{\pa r} \log \vert G(z) \vert \right\vert \leq \left \vert \left( \log G(z) \right)^{\prime} \right \vert = \vert G^{\prime} / G \vert
\]
and $G^{\prime}(z) = (z^{-2} F(z))^{\prime} = z^{-2} F^{\prime}(z)  - 2z^{-3} F(z)$. By the area formula
\[
\int_{D \cap B(0, \rho)} \vert F^{\prime} \vert^2 \leq  C \nu
\]
where $C > 0$ depends only on $\eta_0$. Since $D$ is bounded, it follows that $F^{\prime} \in L^1(D \cap B(0, \rho))$. All other terms in the expression for $G^{\prime}$ are harmless on $D \cap (B(0, \rho) \setminus \ov{B(0, \eta)})$ and hence $\pa / \pa r \log \vert G(z) \vert$ is integrable on $D \cap (B(0, \rho) \setminus \ov{B(0, \eta)})$.

\medskip

To integrate first with respect to $r$, note that the line $\theta = \theta_0$ intersects $D \cap (B(0, \rho) \setminus \ov{B(0, \eta)})$ in at most a countable number of segments. At the end points of each of these segments, $\vert G(z) \vert = \vert S(z)/z \vert = \vert \ov z / z \vert = 1$, except when $\eta e^{i \theta_0}$ or $\delta e^{i \theta_0}$ are in $D$. In the former case, $\log \vert G(z) \vert = 0$, while in the latter case
\[
-2\ep \le \log(1 - \ep) \le \log \vert G(z) \vert \le \log(1 + \ep) \le 2 \ep
\]
if $\ep$ is small enough. Hence
\[
\left\vert \int_{\eta}^{\rho} \frac{\pa}{\pa r} \log \vert G(z) \vert \; dr \right \vert \le 4 \ep
\]
which shows that
\[
\left \vert  \int_{\eta}^{\rho} \frac{A(r)}{r} \;dr \right \vert \le 8 \pi \ep
\]
for every $0 < \eta < \rho$. To complete the proof, suppose that $\nu > 2$ which means that $A(r) > 0$ for all small $r$. Therefore,
\[
\int_{\eta}^{\rho} \frac{A(r)}{r} \; dr = \left \vert  \int_{\eta}^{\rho} \frac{A(r)}{r} \;dr \right \vert \le 8 \pi \ep.
\]
But then
\[
(2 \pi \nu - 4 \pi) \int_{\eta}^{\rho} \frac{dr}{r} \le \int_{\eta}^{\rho} \frac{A(r)}{r} \; dr \le 8 \pi \ep
\]
and hence
\[
(2 \pi \nu - 4 \pi)(\log \rho -\log \eta) \le 8 \pi \ep
\]
for all $0 < \eta < \rho$. This is a contradiction and so $\nu \le 2$.

\end{proof}


\section{Regularity of the boundary: the case $\nu = 1$}

\no In this case, the map $F$ in (3.3) is a biholomorphism. Also, $F|_{\pa D \cap B(0, \delta)} \geq 0$. Let $\sqrt{F}$ be a biholomorphic map from $D \cap B(0, \delta)$ onto the half disc
\[
H^+ = \{ w \in B(0, \sqrt{\eta_0}) : \Im w > 0\}
\]
and let $z = T_+(w) : H^+ \ra D \cap B(0, \delta)$ be its inverse. Then, the cluster set of the interval $I_{\eta_0} = (-\sqrt{\eta_0}, \sqrt{\eta_0}) \subset \pa H^+$ under $T_+$ is contained in $\pa D \cap B(0, \delta)$. In fact, by Proposition 3.2, $T_+$ extends continuously to $I_{\eta_0}$. Define
\[
T_-(w) = \ov{S \circ T_+(\ov w)}
\]
on the lower half disc
\[
H^- = \{ w \in B(0, \sqrt{\eta_0}) : \Im w < 0\}.
\]
\no {\it Claim:} The map
\[
\ti T(\tau) =
\begin{cases}
T_+(w),  &\text{if $w \in H^+$}\\
T_-(w),  &\text{if $w \in H^-$}
\end{cases}
\]
extends to an injective holomorphic map on $B(0, \sqrt{\eta_0})$. To show this, it suffices to prove that $T_-$ extends continuously to $I_{\eta_0}$ and that $T_+ = T_-$ there. For this, fix an arbitrary $x \in I_{\eta_0}$ and note that as $w \ra x$, $w \in H^-$, $T_+(\ov w)$ has a well defined limit, say $\zeta_0 \in \pa D \cap B(0,  \delta)$. Hence, $\ov{S(\zeta_0)} = {\zeta_0}$ and this implies that
\[
T_-(w) = \ov{S \circ T_+(\ov w)} \ra \ov{S(\zeta_0)} = \zeta_0
\]
as $w \ra x$. Thus, $T_+ = T_-$ on $I_{\eta_0}$ and this means that $\ti T$ extends holomorphically to $B(0, \sqrt{\eta_0})$. The injectivity of $\ti T$ follows in a similar manner. Indeed, fix $x \in I_{\eta_0}$ and let $T_+(w) \ra \zeta_0 \in \pa D \cap B(0, \delta)$ as $w \ra x$, as before. Then
\begin{equation}
\vert x \vert ^2 = \lim_{w \ra x} \vert w \vert^2 = \lim_{w \ra x} \vert (T_+)^{-1}(z) \vert^2 = \lim_{w \ra x} \vert zS(z) \vert = \vert \zeta_0 S(\zeta_0) \vert = \vert \zeta_0 \vert^2 = \vert T_+(x) \vert^2.
\end{equation}
Since $T_+ = \ti T$ on $I_{\eta_0}$ and $\vert T_+(x) \vert = \vert x \vert$, it follows that $\ti T$ is injective and finally, that
\[
\Gamma = \pa D \cap B(0, \delta) = {\ti T}^{-1}(I_{\eta_0})
\]
is a smooth real analytic arc passing through the origin, in this case. This is Case $(1)$ of Sakai's theorem.


\section{Regularity of the boundary: the case $\nu = 2$}

\no In this case, an application of the generalized argument principle (see Corollary 4.2 in \cite{S1}) shows that
\[
\int_{\gamma} d \arg F(z) = 2 \pi \nu n (\gamma; 0) = 4 \pi n(\gamma; 0)
\]
where $\gamma$ is a smooth closed path in $D \cap B(0, \delta)$ and $n(\gamma; 0)$ is its winding number around the origin. This means that
\[
\frac{1}{2} \int_{\gamma} d \arg F(z) = \text{a multiple of} \; 2 \pi
\]
and thus $\pm \sqrt{F(z)} = \pm \sqrt{z S(z)}$ are well defined holomorphic functions on $D \cap B(0, \delta)$. By Proposition $3.2$, both have non-vanishing derivative and they are therefore injective as well on $D \cap B(0, \delta)$. To take them both into account, consider
\[
V = \{(z, w) \in (D \cap B(0, \delta)) \times \mbb C^2 : w^2 = z S(z)\}
\]
which is a pure $1$-dimensional analytic set in $(D \cap B(0, \delta)) \times \mbb C^2$. Let $\pi_1, \pi_2$ be the projections on the $z, w$ variables respectively. The projection $\pi_2 : \ov V \ra \mbb C_w$ is proper; in fact, $\pi_2^{-1}(0) \cap \ov V$ contains $(0,0)$ as an isolated point. Moreover, the limit points of $V$ are contained in
\[
\Si = \{ (z, w) \in \ov V, z \in \Gamma \},
\]
which can be alternately described as follows. On $\Gamma$, $S(z) = \ov z$ and so $w^2 = z S(z) = \vert z \vert^2$. Writing $w =  u+iv$ and $z = x+iy$, this is equivalent to the system of equations
\[
uv = 0, \; u^2 - v^2 = x^2 + y^2.
\]
When $u = 0$, then $x^2 + y^2 = -v^2$ and thus the only solution is the point $(0, 0) \in \ov V$. On the other hand, when $u \not= 0$, then $v = 0$ and thus $u^2 = x^2 + y^2$. To summarize,
\[
\Si = \{(z, w) : u^2 = x^2 + y^2, v = 0 \},
\]
and a calculation shows that $\Si \sm \{(0, 0)\}$ is a smooth maximally totally real manifold in $\mbb C^2$. The Alexander--Becker theorem shows that $\ov V \cap (\Si \sm \{(0,0)\})$ is a locally finite union of smooth real analytic arcs and points. The remarks made between Propositions $3.1$ and $3.2$ apply here as well; some of the smooth arcs in $\ov V \cap (\Si \sm \{(0,0)\})$ are those across which the branches $\sqrt{F(z)}$ are single-valued, while there are others across which $V$ is completed by adding to it its reflection near points of $\Si \sm \{(0,0)\}$.

\medskip

Now fix a branch of the square root of $F(z)$, which we will denote by $\sqrt{F}$, and let $T(w) = (\sqrt{F})^{-1}(w) = \text{some branch of} \;\pi_1 \circ \pi_2^{-1}(w)$. Note that $T$ is holomorphic on $B(0, \sqrt{\eta_0}) \sm I$, for some closed subset $I$ contained in the interval $(-\sqrt{\eta_0}, \sqrt{\eta_0})$. The set $I$ arises as the projection on the $w$-variable of some arcs in $\ov V \cap (\Si \sm \{(0,0)\})$ across which $\sqrt{F(z)}$ is not single-valued. If $\ti V$ be the analytic continuation of $V$ across $\Si \sm \{(0,0)\}$, the projection $\pi_2 : \ti V \ra \mbb C_w$ has discrete fibres and its branch locus is zero-dimensional in $\Si \sm \{(0,0)\}$. It follows that for $u \in I$ away from the branch locus,
\[
\lim_{v \ra 0^+} T(u + iv) = \vert u \vert e^{i \theta_+(u)}
\]
exists; the fact that it equals the given value follows from an argument similar to $(4.1)$ and the description of $\Si$.

\medskip

Since $\ov{T(\ov w)}$ is bounded, Fatou's theorem shows that
\[
\lim_{v \ra 0^-} T(u + iv) = \vert u \vert e^{i \theta_-(u)}
\]
exists for almost every $u \in I$ and hence so does
\[
\lim_{v \ra 0^{\pm}} T(u + iv) = \vert u \vert e^{i \theta_{\pm}(u)}
\]
for almost every $u \in I$. Note that in both cases, the modulus of the limiting values of $T$ equals $\vert u \vert$. There are three cases to consider:

\medskip

\no {\sf Case A:} Suppose that $\vert u \vert e^{i \theta_+(u)} = \vert u \vert e^{i \theta_-(u)}$ for almost every $u \in I$. Then $T$ extends to a holomorphic function on (a possibly smaller disk around the origin in) $B(0, \sqrt{\eta_0})$. The extension is injective near the origin since the origin is not an isolated point in $I$ and $\vert T(u) \vert = \vert u \vert$ for almost every $u \in I$. In fact, $T(B(0, \sqrt{\eta_0}) \sm I) = D \cap B(0, \delta)$ and it maps almost every point on $I$ to $\Gamma$. The closure of $\Gamma$, in this case, uniquely determines a smooth real analytic curve, namely $T((-\sqrt{\eta_0}, \sqrt{\eta_0}))$. This is Case $(2a)$ of Sakai's theorem.

\medskip

Let
\[
P = \left\{ u \in I: \lim_{v \ra 0^{\pm}} T(u + iv) = \vert u \vert e^{i \theta_{\pm}}(u) \; \text{but} \; \vert u \vert e^{i \theta_+(u)} \not= \vert u \vert e^{-i \theta_-(u)} \right\}.
\]
As discussed above, Case A treats the situation when $P$ has zero measure and therefore we may assume that $P$ has positive measure in every interval around the origin. At $u \in P$, let us also note that the function
\[
\psi(w) = T(w) \ov{T(\ov w)}/ w^2
\]
which is holomorphic on $B(0, \sqrt{\eta_0}) \sm I$, satisfies
\[
\lim_{v \ra 0^{\pm}} \psi(u + iv) = \vert u \vert^2 e^{\pm i(\theta_+(u) - \theta_-(u))}/ u^2 = e^{\pm i(\theta_+(u) - \theta_-(u))} \not= 1.
\]
Hence, $\psi \not\equiv 1$. Now consider the set of all circles $C_{x, r}$ centered at points $x \in (-\sqrt{\eta_0}, \sqrt{\eta_0})$ with radius $r > 0$ such that $C_{a, r} \subset B(0, \sqrt{\eta_0})$. Since $P$ has positive measure near the origin, the uniqueness theorem for holomorphic functions shows that there exists a circle in this family, say $C_{a, \rho}$ such that the two diametrically opposite points $a \pm \rho \in P$ and $\psi \not= 1$ at all points of $C_{a, \rho}$. We may assume that $C_{a, \rho}$ contains the origin in its interior. Consider the restriction of $\psi$ to $B(a, \rho) \sm I$, the goal being to determine the cluster set of $C_{a, \rho} \cup I$ under $\psi$.

\medskip

First, $\psi$ maps $C_{a, r} \sm \{a \pm \rho\}$ to a real analytic arc, say $J$ and by the observations made above, its end-points stay away from $z = 1$. At points $u \in I \sm \{(0,0)\}$, $\vert \psi(u) \vert = 1$. Finally, as $w \ra (0,0)$, $w \in B(a, \rho) \sm I$,
\[
\psi(w) = T(w) \ov{T(\ov w)}/ w^2 = \left( T(w)/w \right) \ov{\left( T(\ov w)/\ov w \right)} \ra 1
\]
due to two reasons. One, the definition of $T(w)$ shows that $w = \sqrt{F}(T(w)) = \sqrt{T(w) S(T(w))}$ and hence
\[
T(w)/ w = T(w) / \sqrt{T(w) S(T(w))} = \sqrt{T(w) / S(T(w))}.
\]
Secondly, $T(w) \ra 0$ as $w \ra 0$ (since $\pi_2^{-1}(0) \cap \ov V = (0,0)$) and therefore by Proposition $3.1$, $\psi \ra 1$ as $w \ra 0$, $w \in B(a, \rho) \sm I$. The conclusion of all this is that the cluster set of $C_{a, \rho} \cup I$ under $\psi$ is contained in $J \cup \{\vert z \vert = 1\}$. Put differently,
\[
\psi : B(a, \rho) \sm I \ra \mbb C \sm (J \cup \{\vert z \vert = 1\}
\]
is proper. Also, note that for $u \in (-\sqrt{\eta_0}, \sqrt{\eta_0}) \sm I$, $\psi(u) = \vert T(u) \vert^2/ u^2> 0$.

\medskip

Now, the properness of $\psi$ shows that $I$ has finitely many components near the origin. Indeed, choose a small $\ep > 0$ so that $B(1, \ep) \cap J = \emptyset$. Since $\psi$ is proper, it follows that $\psi^{-1}(B(1, \ep) \sm \{ \vert z \vert = 1 \})$ has finitely many components, say $n$. If $I$ were to have more than $n$ components, say $I_1, I_2, \ldots, I_{n+1}$, we can choose disjoint neighborhoods $U_1, U_2, \ldots, U_{n+1}$ around them respectively so that $\psi \not= 1$ on any $\pa U_i$ and $\psi(\cup \pa U_i) \cap B(1, \ti \ep) = \emptyset$ for some $0 < \ti \ep < \ep$. For each $U_i$, observe that $\psi(u) > 0$ for $u \in (U_i \cap (-\sqrt{\eta_0}, \sqrt{\eta_0})) \sm I$ and $\vert \psi(u) \vert \ra  1$ as $u \ra I$ through $(U_i \cap (-\sqrt{\eta_0}, \sqrt{\eta_0})) \sm I$. Since $\psi(\cup \pa U_i) \cap B(1, \ti \ep) = \emptyset$, it means that each $U_j$ contains at least one component of $\psi^{-1}(B(1, \ti \ep) \sm \{ \vert z \vert = 1 \})$. Thus, the number of connected components of $\psi^{-1}(B(1, \ti \ep) \sm \{ \vert z \vert = 1 \})$ is at least $n+1$, which is a contradiction. Therefore, there is an interval in $I$ containing the origin in its closure. This distinguished interval will still be denoted by $I$. There are two cases to consider:

\medskip

\no {\sf Case B:} The origin is an interior point of $I$. In this case, we may take $I = (-\sqrt{\eta_0}, \sqrt{\eta_0})$. Let
\[
H^{\pm} = \{ w \in B(0, \sqrt{\eta_0}: \Im w > \pm 0 \}
\]
and define
\[
\ti T_1(w) =
\begin{cases}
T(w), &\text{if $w \in H^+$}\\
\ov{S(T(\ov w))}, &\text{if $\ov w \in H^+$}.
\end{cases}
\]
Note that $T(\ov w)$, which occurs in the definition of $\ti T_1$ in the lower half disc, is the reflection of $T(w)$ in the real axis. Then $\ti T_1$ extends holomorphically to $B(0, \sqrt{\eta_0})$. To see this, recall that the set of points $u \in I$ for which
\[
\lim_{v \ra 0^+} T(u + iv)
\]
exists has full measure. Fix $w_0 = u_0 \in I$ from this set. As $w \in H^+$ converges to $w_0$, let $T(w) \ra \zeta_0 \in \pa D \cap B(0, \delta)$. As before, $\vert T(w_0) \vert = \vert w_0 \vert$. Now, let $w_j \in H^-$ converge to $w_0$ and let $T(\ov w_j) = \ti \zeta_j$. Note that $\ti \zeta_j \ra \zeta_0$ since $\ov w_j \ra w_0$. Then
\[
\ov w_j = T^{-1}(\ti \zeta_j)
\]
shows that
\[
\ov{S(T(\ov w_j))} = \ov{S(\ti \zeta_j)} \ra \ov{S(\zeta_0)} = \zeta_0.
\]
Thus, $\ti T_1$ extends holomorphically to $B(0, \sqrt{\eta_0})$. The extension is injective near the origin since $\vert T(w_0) \vert = \vert w_0 \vert$ for almost every $w_0 \in I$. Similarly,
\[
\ti T_2(w) =
\begin{cases}
\ov{S(T(\ov w))}, &\text{if $w \in H^+$}\\
T(w), &\text{if $\ov w \in H^+$}
\end{cases}
\]
extends to a holomorphic map on $B(0, \sqrt{\eta_0})$ and the extension is injective near the origin. The extensions $\ti T_1, \ti T_2$ are not the same, as otherwise
\[
\lim_{v \ra 0^+} T(u + iv) = \lim_{v \ra 0^-} T(u+iv)
\]
on $I$ and this violates the assumption that $P$ has positive measure. In this situation, $D \cap B(0, \delta)$ consists of two components each of which arises as the image of $H^+$ under $\ti T_1, \ti T_2$. The boundaries of these components are precisely $\ti T_1(I)$ and $\ti T_2(I)$ respectively and are hence smooth real analytic arcs. Since $\pi_2^{-1}(0) \cap \ov V = (0, 0)$, it follows that $\ti T_1(0) = \ti T_2(0) = 0$ and hence both arcs pass through the origin. To see that $\ti T_1'(0) = \ti T_2'(0)$, note that for $w \in H^+$ and $\ti T_1(w) = \zeta$,
\[
(\ti T_1(w) - \ti T_1(0))/w = \ti T_1(w)/w = \zeta/w = \zeta / (\ti T_1)^{-1}(\zeta) = \zeta / \sqrt{\zeta S(\zeta)}.
\]
Similarly, for $\ov w \in H^+$ and $\ti T_2(w) = T(w) = \ti \zeta$,
\[
(\ti T_2(w) - \ti T_2(0))/w = \ti T_2(w)/w = \ti \zeta / w = \ti \zeta / T^{-1}(\ti \zeta) = \ti \zeta / \sqrt{\ti \zeta S(\ti \zeta)}.
\]
Since both $\zeta, \ti \zeta \ra 0$ as $w \ra 0$, it follows from Proposition $3.1$ that $\ti T_1'(0) = \ti T_2'(0)$. This is Case $(2b)$ of Sakai's theorem.

\medskip

\no {\sf Case C:} The origin is an end point of $I$. In this case, we may assume that $I = [0, \sqrt{\eta_0})$. The argument is similar in spirit to the case when $\nu = 1$. Let $\sqrt[4]{F}$ (which is the same thing as $\sqrt{T}$) be chosen so that it maps $D \cap B(0, \delta)$ biholomorphically to $H^+$, where
\[
H^+ = \{w \in B(0, \sqrt[4]{\eta_0}): \Im w > 0\}
\]
and let $\ti T(w) = (\sqrt[4]{F})^{-1}(w)$. As before, it can be checked that
\[
T_{\ast}(w) =
\begin{cases}
\ti T(w), &\text{if $w \in H^+$}\\
\ov{S(\ti T(\ov w))}, &\text{if $\ov w \in H^+$}
\end{cases}
\]
extends holomorphically to $B(0, \sqrt[4]{\eta_0})$. By the definition of $\ti T$, it follows that if $\ti T(w) = z$ then
\[
w = (\ti T)^{-1}(z) = \sqrt[4]{z S(z)}.
\]
Taking limits as $w \ra w_0 \in (-\sqrt[4]{\eta_0}, \sqrt[4]{\eta_0})$ and $z \ra z_0 \in \pa D \cap B(0, \delta)$, we see that
\[
\vert w_0 \vert^4 = z_0 S(z_0) = \vert z_0 \vert^2
\]
which implies that $\vert \ti T(w_0) \vert = \vert w_0 \vert^2$. This shows that the extended map $\ti T$ is not injective near the origin; in fact, it has a zero of order two at the origin. Further, as observed earlier, $\ti T$ is injective on $H^+$. Suppose that $\ti T$ fails to be injective on $\ov H^+$ near the origin. In this case, there will exist $\al_j, \beta_j \in (-\sqrt[4]{\eta_0}, \sqrt[4]{\eta_0})$ such that both $\al_j, \beta_j \ra 0$ and $\ti T(\al_j) = \ti T(\beta_j)$. Then
\[
\vert \ti T(\al_j) \vert = \vert \ti T(\beta_j) \vert = \vert \al_j \vert^2 = \vert \beta_j \vert^2
\]
and hence $\al_j = \pm \beta_j$. The only possibility then is that $\al_j = - \beta_j$ for all $j$. By the uniqueness theorem, $\ti T(x) = \ti T(-x)$ for all $x \in (-\sqrt[4]{\eta_0}, \sqrt[4]{\eta_0})$ and this implies that
\[
\lim_{v \ra 0^+} T(u + iv) = \lim_{v \ra 0^-} T(u+iv)
\]
on $I$. This is a contradiction. The remaining possibility is that $\alpha_j \in (-\sqrt[4]{\eta_0}, \sqrt[4]{\eta_0})$, $\beta_j \in H^+$ and $\ti T(\al_j) = \ti T(\beta_j)$. We may assume that $T'_{\ast}(w) \not= 0$ for $w \not= 0$. Then, $\ti T$ is a locally injective near $\al_j$ and $\beta_j$. As $\ti T(\al_j) = \ti T(\beta_j)$ by assumption, there will exist points $a, b \in H^+$ near $\al_j, \beta_j$ respectively such that $\ti T(a) = \ti T(b)$ and this contradicts the injectivity of $\ti T$ on $H^+$. This is Case $(3)$ of Sakai's theorem.

\medskip

We conclude by revisiting the discussion about the valency of $S(z)$ that was left incomplete in Section $2$. The remaining case, as mentioned there, is when $w \in \pa D$. Let $w = 0$ for brevity. There are two cases to consider. First, if $w = 0$ is either regular or degenerate, then $S$ is locally invertible near all points in $S^{-1}(0)$ as the discussion in these cases shows. Also, there is a unique point in $A$ that lies over $w = 0$, namely $(0, S(0)) = (0, 0)$ and hence there must be exactly $N-1$ points in $A^{\ast}$ that lie over $w = 0$. Thus, the valency of $S$ is $N-1$. Second, if $w = 0$ is a double point, then $D$ is locally the union of two components bounded by smooth real analytic arcs that are tangent to each other at $w = 0$. The Schwarz function $S(z)$, restricted to each of these components admits an injective holomorphic extension across the origin. The two extensions are distinct as germs but each satisfies $S(0) = 0$. Again, there is a unique point in $A$ that lies over $w=0$, namely $(0, S(0)) = (0, 0)$. Hence, there must be exactly $N-1$ points in $A^{\ast}$ that lie over $w=0$ and this means that the valency is again $N-1$. Lastly, when $w=0$ is a cusp, the discussion in Case $(3)$ shows that $S(z) = 0$ has a double root at the origin. Then there are two points in $A$ and hence exactly $N-2$ points in $A^{\ast}$ that lie over $w=0$. In this case, the valency is $N-2$.


\section{Concluding Remarks}

\no Working with the graph of $S(z)$ and appealing to the Alexander--Becker theorem shows that a classical quadrature domain can be realised as an open subset of a pure $1$-dimensional analytic set in $\mbb P^1 \times \mbb P^1$. Earlier work of Gustafsson--Putinar \cite{GP1} shows that such quadrature domains admit a canonical rational embedding in the exterior of the unit ball in $\mbb C^N$, $N$ being the order of the quadrature domain. Their work is perhaps the closest in spirit to the presentation here. In fact, in Section $5$, pp. $208$ of \cite{GP1}, they write that {\it there are also other ways to inject quadrature domains into projective spaces so that the Schwarzian reflection becomes a geometric object} and go on to remark that {\it perhaps the most natural way is via the injection $z \mapsto (z, S(z))$}. In hindsight, this work can be seen as an exploration of this comment. Another approach to the boundary regularity question can be found in \cite{GP2}.

\medskip

It is also possible to study more general quadrature domains that are associated with complex measures on the plane. Sakai \cite{S2} has a boundary regularity theorem for such quadrature domains as well and as an application of such results, we refer the reader to \cite{S3}. Providing an alternate proof of this more general regularity theorem within the framework discussed here presents new difficulties and addressing them seems an interesting enterprise that we leave for the future.


\end{document}